\documentclass[a4paper,11pt]{amsart}
\usepackage[left=2.7cm,right=2.7cm,top=3.5cm,bottom=3cm]{geometry}
\usepackage{amsthm,amssymb,amsmath,amsfonts,mathrsfs,amscd,amsbsy,dsfont,verbatim,mathtools}
\usepackage[utf8]{inputenc}
\usepackage[all,cmtip]{xy}
\usepackage{amssymb}
\usepackage{amscd,amsfonts,amsmath,amssymb, bbm,dsfont, comment}
\usepackage{amsmath}
\usepackage{color}
\usepackage{amsfonts}
\usepackage{amssymb}\usepackage{enumerate,epsf,fancyhdr,float,graphicx,tabularx}
\usepackage{latexsym,mathrsfs,multirow}
\usepackage{wasysym}
\usepackage{xypic}
\usepackage[all]{xy}\usepackage[OT2,T1]{fontenc}
\usepackage[modulo,mathlines,displaymath, running]{lineno}
\usepackage{amsmath,mathrsfs,enumerate,wasysym}
\usepackage{xypic,hhline}
\usepackage[all]{xy}
\usepackage[OT2,T1]{fontenc}
\usepackage[modulo,mathlines,displaymath]{lineno}
\usepackage{tikz,tikz-cd}
\usepackage{dashrule}
\usetikzlibrary{matrix}
\usepackage[modulo,mathlines,displaymath]{lineno}

\newtheorem{theorem}{Theorem}[section]

\DeclareSymbolFont{cyrletters}{OT2}{wncyr}{m}{n}\DeclareMathSymbol{\Sha}{\mathalpha}{cyrletters}{"58}
\DeclareMathSymbol{\FSha}{\mathalpha}{cyrletters}{"11}
\renewcommand{\phi}{{\varphi}}

\renewcommand{\geq}{\geqslant}
\renewcommand{\leq}{\leqslant}

\newcommand{\links}{\left(\begin{array}{cc}}
\newcommand{\rechts}{\end{array}\right)}
\newcommand{\bai}{\left[\begin{array}{cc}}
\newcommand{\dai}{\end{array}\right]}
\newcommand{\hidari}{\left(\begin{array}{c}}
\newcommand{\migi}{\end{array}\right)}

\newcommand{\C}{\mathbb{C}}

\newcommand{\N}{\mathbb{N}}

\newcommand{\Q}{\mathbb{Q}}

\newcommand{\Z}{\mathbb{Z}}

\newcommand{\calE}{\mathcal{E}}

\DeclareMathOperator{\tr}{tr}

\newcommand{\Gal}{\operatorname{Gal}}

\newcommand{\Hom}{\operatorname{Hom}}
\newcommand{\res}{\operatorname{res}}

\newcommand{\Aut}{\operatorname{Aut}}

\newcommand{\GL}{\operatorname{GL}}

\newcommand{\Sel}{\operatorname{Sel}}

\renewcommand{\contentsname}{Contents\\{\footnotesize\normalfont(A table
of contents should normally not be included)}}

\theoremstyle{plain}

\newtheorem*{ThmA}{Theorem A}

\newtheorem{auxiliary proposition}[theorem]{Auxiliary Proposition}

\newtheorem{corollary}[theorem]{Corollary}
\newtheorem{lemma}[theorem]{Lemma}
\newtheorem{proposition}[theorem]{Proposition}

\theoremstyle{definition}
\newtheorem{definition}[theorem]{Definition}

\newtheorem{main conjecture}[theorem]{Main Conjecture}
\newtheorem{main theorem}[theorem]{Main Theorem}
\newtheorem{modesty proposition}[theorem]{Modesty Proposition}

\newtheorem{open problem}[theorem]{Open Problem}

\theoremstyle{remark} 
\newtheorem{remark}[theorem]{Remark}

\newtheorem{convergence lemma}[theorem]{Convergence Lemma}
\newtheorem{corrected lemma}[theorem]{Corrected Lemma}
\newtheorem{growth lemma}[theorem]{Growth Lemma}
\newtheorem{coefficient lemma}[theorem]{Integrality Lemma}
\newtheorem{interpolation lemma}[theorem]{Interpolation Lemma}
\newtheorem{kernel lemma}[theorem]{Kernel Lemma}
\newtheorem{limit lemma}[theorem]{Limit Lemma}
\newtheorem{tandem lemma}[theorem]{Modesty Lemma}
\newtheorem{zero-finding lemma}[theorem]{Zero-Finding Lemma}

\newcommand{\longmono}{\mbox{\;$\lhook\joinrel\longrightarrow$\;}}

\newcommand{\longepi}{\mbox{\;$\relbar\joinrel\twoheadrightarrow$\;}}

\newcommand{\Ta}{\mathrm{Ta}}

\newcommand{\defeq}{\vcentcolon=}

\usepackage{graphicx}
\newcommand{\Bmu}{\mbox{$\raisebox{-0.59ex}
  {$l$}\hspace{-0.18em}\mu\hspace{-0.88em}\raisebox{-0.98ex}{\scalebox{2}
  {$\color{white}.$}}\hspace{-0.416em}\raisebox{+0.88ex}
  {$\color{white}.$}\hspace{0.46em}$}{}}
  
\makeatletter
\@namedef{subjclassname@2020}{%
  \textup{2020} Mathematics Subject Classification}
\makeatother

\numberwithin{equation}{section}

\begin{document}

\title[On anticyclotomic Selmer groups of elliptic curves]{On anticyclotomic Selmer groups of elliptic curves}

\author{Matteo Longo, Jishnu Ray and Stefano Vigni}
\address{Dipartimento di Matematica, Universit\`a di Padova, Via Trieste
63, 35121 Padova, Italy.}
\email{matteo.longo@unipd.it}
\address{
	Harish Chandra Research Institute, A CI of Homi Bhabha National
	Institute, Chhatnag Road, Jhunsi, Prayagraj (Allahabad), 211 019 India.
}
\email{jishnuray@hri.res.in}
\address{Dipartimento di Matematica, Universit\`a di Genova, Via Dodecaneso 35, 16146 Genova, Italy.}
\email{stefano.vigni@unige.it}

\thanks{The first and third authors are partially supported by PRIN 2022 ``The arithmetic of motives and $L$-functions'' and by the GNSAGA group of INdAM. The second author gratefully acknowledges support of Inspire Research Grant, Department of Science and Technology, Government of India. The research by the third author is partially supported by the MUR Excellence Department Project awarded to Dipartimento di Matematica, Universit\`a di Genova, CUP D33C23001110001.}

%\date{~}

\begin{abstract}
Let $p\geq5$ be a prime number and let $K$ be an imaginary quadratic field where $p$ is unramified. Under mild technical assumptions, in this paper we prove the non-existence of non-trivial finite $\Lambda$-submodules of Pontryagin duals of signed Selmer groups of a $p$-supersingular rational elliptic curve over the anticyclotomic $\Z_p$-extension of $K$, where $\Lambda$ is the corresponding Iwasawa algebra. In particular, we work under the assumption that our plus/minus Selmer groups have $\Lambda$-corank $1$, so they are not $\Lambda$-cotorsion. Our main theorem extends to the supersingular case analogous non-existence results by Bertolini in the ordinary setting; furthermore, since we cover the case where $p$ is inert in $K$, we refine previous results of Hatley--Lei--Vigni, which deal with $p$-supersingular elliptic curves under the assumption that $p$ splits in $K$.
\end{abstract}

\maketitle

\section{Introduction} 

Let $p$ be an odd prime and let $K$ be an imaginary quadratic field where $p$ is unramified. The goal of this article is to complete the work of Bertolini (\cite{BerBor}) and of Hatley--Lei--Vigni (\cite{HLV2022}) on the non-existence of non-trivial finite $\Lambda$-submodules of Pontryagin duals of Selmer groups of rational elliptic curves over the anticyclotomic $\Z_p$-extension of $K$, where $\Lambda$ is the corresponding Iwasawa algebra; in those two papers the case of $\Lambda$-corank $1$ of the relevant Selmer groups is considered. More precisely, the paper by Bertolini covers the case of elliptic curves with good ordinary reduction at $p$, while the work by Hatley--Lei--Vigni extends the techniques to $p$-supersingular elliptic curves when $p$ splits in $K$. Our main aim is to generalize the aforementioned results to $p$-supersingular elliptic curves when the prime $p$ is inert in $K$, hence completing the picture. Here we wish to emphasize that what makes it possible for us to attack this case are recent results of Burungale--Kobayashi--Ota (\cite{BKO2021}) proving Rubin's conjecture on the structure of local units in the anticyclotomic $\Z_p$-extension of the unramified quadratic extension of $\Q_p$ for a prime $p\geq5$ (\cite{rubin87}). We also take the occasion to refine some of the results in \cite{HLV2022}. 

In order to describe our main result, we need to introduce some notation and our running assumptions. Let $E$ be an elliptic curve over $\Q$ without complex multiplication. Let $N$ be the conductor of $E$ and let $K$ be an imaginary quadratic field of discriminant $-D_K$ coprime to $N$, whose ring of integers will be denoted by $\mathcal O_K$. Write $N=N^+N^-$ so that all prime factors of $N^+$ split in $K$ and all prime factors of $N^-$ are inert in $K$. We assume throughout that $N^-$ is a square-free product of an \emph{even} number of primes, which means that we are in the \emph{indefinite} setting. Let $h_K$ be the class number of $K$ and let $p\geq5$ be a prime number such that $p\nmid ND_Kh_K$; in particular, $E$ has good reduction at $p$. For each integer $m\geq 0$, let $H_{p^m}$ be the ring class field of $K$ of conductor $p^m$; in particular, if $m=0$, then $H_1$ is the Hilbert class field of $K$, which we also denote by $H_K$. Let $K_\infty$ be the anticyclotomic $\Z_p$-extension of $K$, which we realize as a subfield of a fixed algebraic closure of $K$; for all $m\geq0$, let $K_m$ be the $m$-th layer of $K_\infty/K$, \emph{i.e.}, the unique subextension of $K_\infty/K$ such that $[K_m:K]=p^m$. Since $p\nmid h_K=[H_K:K]$, there is an isomorphism $\Gal(H_{p^m}/K)\simeq\Gal(K_{m-1}/K)\times\Delta$ where $\Delta\defeq\Gal(H_p/K)$; in particular, $K_0=K$ and $\#\Delta=h_K\cdot\bigl(p-\varepsilon_K(p)\bigr)/u_K$, where $u_K\defeq\#\mathcal{O}_K^\times/2$ and $\varepsilon_K$ is the Dirichlet character attached to $K$. Furthermore, $[H_p:H_K]=(p+1)/u_K$ (see, \emph{e.g.}, \cite[Theorem 7.24]{cox}). Finally, set $\Gamma\defeq\Gal(K_\infty/K)$ and denote by $\Lambda\defeq\Z_p[\![\Gamma]\!]$ the Iwasawa algebra of $\Gamma$.

If $p$ is ordinary for $E$, then the Pontryagin dual $\mathfrak{X}_p(E/K_\infty)$ of the $p$-power Selmer group $\Sel_{p^\infty}(E/K_\infty)$ is a $\Lambda$-module of rank $1$; 
the reader is referred to \cite[Theorem A]{Bertolini95}, which holds under the assumptions described in \cite[\S 2.2]{Bertolini95}. The assumptions in \cite{Bertolini95} have been subsequently relaxed and, to the best of our knowledge, \cite[Theorem 3.2]{Nek07} together with \cite[\S2]{Nek01} offer the most general version of this result (see also \cite[Theorem A]{BerLongVen} and \cite[Theorem B]{Ho04}). From here on, we assume that $p$ is \emph{supersingular} for $E$. Then one may define signed Selmer groups $\Sel_{p^\infty}^\pm(E/K_\infty)$, depending on the choice of a sign $\pm$, and their Pontryagin duals $\mathfrak{X}_p^\pm(E/K_\infty)$ are again (at least under mild technical assumptions) $\Lambda$-modules of rank $1$. We remark that if $N^-$ is a square-free product of an \emph{odd} number of primes, then the signed Selmer groups are $\Lambda$-cotorsion: this situation is handled, \emph{e.g.}, in a recent paper by Shii (\cite{Shii}). 

%We now list the technical assumptions required to prove the main result of this paper (in addition to those listed before). 
%\begin{assumption} Let $E/\Q$ be an elliptic curve of conductor $N$, $K/\Q$ a quadratic imaginary field of 
%discriminant $-D_K$ and class nunber $h_K$, and $p$ a prime number. 
%\begin{enumerate}
%    \item  $(D_K,N)=1$; $p\nmid ND_Kh_K$; writing $N=N^+N^-$ as before, $N^-$ is a square free product of an even number of primes; 
%   \item The residual representation $\bar\rho_{E,p}:G_\Q\rightarrow \Aut(E[p])$ 
%is surjective. 
%\item The prime $p$ does not divide the Tamagawa number $c_\ell=|(E(\Q_\ell)/E_0(\Q_\ell)|$ for all $q\mid N$.
%\item The local norm maps $E(K_{m,v})\rightarrow E(K_w)$ are surjective for all $m\geq 0$ and all primes ideals %$v,w\neq p$
%\end{enumerate}
%\end{assumption}

%\begin{equation}\label{Tor}
%\tag{Tor}
%E(K_{\infty,\gq})[p^\infty]=0 \text{ for any prime $\gq \mid p$}
%\end{equation}

%We remark that (\ref{Tor}) is true %if we assume that 

Now we list some assumptions that will be used later in this paper. We first note that, since $E$ has supersingular reduction at $p$, the representation $\bar\rho_{E,p}:G_\Q\rightarrow\Aut(E[p])$ on the $p$-torsion $E[p]$ of $E$ is irreducible. 
%For this paper we shall assume the stronger assumption that 
%\begin{equation}\tag{Surj}\label{Surj}\bar\rho_{E,p}:G_\Q\longrightarrow \Aut(E[p])\end{equation}
%is surjective. 
For each prime number $\ell$, let $c_\ell(E)\defeq\big|E(\Q_\ell)/E_0(\Q_\ell)\big|$ be the Tamagawa number of $E$ at $\ell$. The first assumption we impose is
\begin{equation}\label{Tam}
\tag{Tam} \text{$p\nmid c_\ell(E)$ for all primes $\ell\,|\,N$.}
\end{equation}
Fix an integer $m\geq0$, let $v$ be a prime of $K$ and let $w$ be a prime of $K_m$ above $v$; let $\tr_{K_{m,w}/K_v}:E(K_{m,w})\rightarrow E(K_v)$ be the corresponding local trace map. The second condition that will be in force is
\begin{equation} \label{Norm}
\tag{tr} \text{$\tr_{K_{m,w}/K_v}$ is surjective for all $m$, all $v\neq p$ and all $w$ as above.}
\end{equation}
The irreducibility of $\bar\rho_{E,p}$ combined with \eqref{Tam} is used to prove a perfect control theorem, whereas \eqref{Norm} plays a role in proving a criterion for the non-existence of non-trivial finite $\Lambda$-submodules of Pontryagin duals of certain Selmer groups (Theorem \ref{THM}). Furthermore, \eqref{Norm} implies \eqref{Tam}, so our main results will be stated only under this stronger condition. 

At various steps, we shall also require 
\begin{equation}\label{Ram}
\tag{Ram$(N^-)$} \text{If $\ell\,|\,N^-$ and $\ell^2\equiv 1\pmod{p}$, then $\bar\rho_{E,p}$ is ramified at $\ell$}
\end{equation}
and 
\begin{equation} \label{Image}
\tag{im(5)} \text{If $p=5$, then the image of $\bar\rho_{E,5}$ contains a conjugate of $\GL_2(\mathbb{F}_5)$.}
\end{equation} 
Of course, \eqref{Image} is implied by the stronger  
\begin{equation}\label{Surj}
\tag{surj} \text{$\bar\rho_{E,p}$ is surjective.}
\end{equation} 
Finally, we state  
\begin{equation}\label{Ram+}
\tag{Ram$(N^+)$} \text{If $\ell\,|\,N^+$, then $E(\Q_\ell)[p]=0$ and $\bar\rho_{E,p}$ is ramified at $\ell$.}
\end{equation} 
Under \eqref{Surj} and when $p$ splits in $K$, it is shown in \cite{LV2019} that the rank of $\mathfrak{X}_p^\pm(E/K_\infty)$ over $\Lambda$ is $1$. The proof of this result in \cite{LV2019} can presumably be extended to the inert case; moreover, it is likely that the surjectivity assumption can be relaxed or even dropped if $p\geq7$ and replaced by \eqref{Image} if $p=5$. The same result is proved in \cite{BerLongVen} under \eqref{Surj}, \eqref{Ram} and \eqref{Ram+} both for $p$ split and for $p$ inert, in \cite{BBLII} in the inert case under \eqref{Image} and \eqref{Ram}, and in \cite{BBL2022} in the split case under \eqref{Image}, \eqref{Ram} and a weaker version of \eqref{Ram+} in which only the ramification of $\bar\rho_{E,p}$ is required. Henceforth we assume that 
\begin{equation} \label{rank}
\tag{rank} \text{$\mathfrak{X}_p^\pm(E/K_\infty)$ is a $\Lambda$-module of rank $1$},
\end{equation} 
which holds true in (at least) all of the situations (and under the relative assumptions) listed above. 

Using plus/minus Mordell--Weil groups $\mathbb{M}_n^\pm\subset E(K_m)\otimes_\Z\Q_p/\Z_p$, which inject into the plus/minus Selmer groups $\Sel_{p^\infty}^\pm(E/K_m)$ via Kummer maps,  
one may introduce plus/minus Shafarevich--Tate groups $\Sha_{p^\infty}^\pm(E/K_m)$ as the quotient of $\Sel_{p^\infty}^\pm(E/K_m)$ by $\mathbb{M}_n^\pm$. Let us write
$\Sha_{p^\infty}^\pm(E,K_m/K_{m+1})$ for the kernel of the restriction map 
$\Sha_{p^\infty}^\pm(E/K_m)\rightarrow \Sha_{p^\infty}^\pm(E/K_{m+1})$; we may then consider the condition  
\begin{equation}\label{Sha}
\tag{$\Sha(m)$}  \text{$\Sha_{p^\infty}^\pm(E,K_m/K_{m+1})=0$.}
\end{equation}
The modules $\mathbb{M}_m^\pm$ are equipped with Heegner points of conductor $p^m$ for all $m\geq 0$, and we may consider the 
$\mathbb{F}_p[\![\Gamma]\!]$-modules 
$\mathbb{E}_m^\pm\subset \Sel_p^\pm(E/K_m)$ generated by these points (see \S\ref{sec:Heegner_points} for details). The last assumption we need to introduce is 
\begin{equation}\label{Heegner}
\tag{$\mathrm{Heeg}(m)$}  \text{$\mathbb{E}_m^\pm\neq 0$.}
\end{equation}
Our main result, which corresponds to Theorem \ref{THM1}, is

\begin{ThmA}\label{THMINTRO}
Assume that \eqref{Norm} and \eqref{rank} hold. If $\Sha(m)$ and $\mathrm{Heeg}(m)$ are satisfied for some integer $m\geq 0$, then $\mathfrak{X}_p^\pm(E/K_\infty)$ has no non-trivial finite submodules. 
\end{ThmA}

The proof of this theorem is obtained by studying how the module of universal norms 
$US_p^\pm(E/K)$ sits inside the $p$-adic Tate module $S_p^\pm(E/K)$ of $\Sel_{p^\infty}^\pm(E/K)$ (see Definition \ref{def:uninorm}). More precisely, one shows that $US_p^\pm(E/K)$ is a free $\Z_p$-module of rank $1$ and the quotient $S_p^\pm(E/K)\big/US_p^\pm(E/K)$ is torsion-free, from which a formal argument yields Theorem A (see Theorem \ref{THM1} for details). As will be apparent, our proofs follow quite closely the arguments in \cite{Ber-Annihilator}, \cite{Bertolini95}, \cite{BerBor} in the ordinary case and the generalizations to the split supersingular case in \cite{HLV2022}.

Now we would like to discuss the relation of the present article with \cite{HLV2022}. Suppose that $p$ splits in $K$. The analogous result on the non-existence of non-trivial finite $\Lambda$-submodules of $\mathfrak{X}_p^\pm(E/K_\infty)$ is stated in \cite{HLV2022} under the stronger condition that $\mathbb{E}_m^\pm$ is non-trivial \emph{for all} $m\geq 0$. Under this assumption, it is possible to show that the module of universal norms $US_p^\pm(E/K)$ is free of rank $1$ over $\Z_p$ and is generated by a point of $E(K)\otimes_\Z\Z_p$ that is not divisible by $p$. The same result holds in the ordinary case under the assumption that $\mathbb{E}_m\neq 0$ for \emph{some} integer $m\geq 0$ (\cite{BerBor}); however, the proof of this result in the ordinary case crucially exploits the fact that the local norm maps are surjective also at $p$, which is false in the case of supersingular reduction, and this is why in \cite{HLV2022} the authors are forced to assume this stronger condition. However, if we only assume $\mathbb{E}_m^\pm\neq 0$ for \emph{some} integer $m\geq 0$ (as in \cite{BerBor}), then we can still prove that $US_p^\pm(E/K)\simeq\Z_p$ and $S_p^\pm(E/K)/US_p^\pm(E/K)$ is torsion-free, although we can no longer show that the universal norms are generated by a point in $E(K)\otimes_\Z\Z_p$, as $\mathbb{E}_0^+=0$ in this case. 

Still in the split supersingular case, we finally remark that the assumption $\mathbb{E}_m^\pm\neq0$ for all $m\geq 0$ is equivalent to the assumption $\mathbb{E}_0^\pm\neq0$: this is accounted for by the trace relations satisfied by Heegner points and recalled in \S\ref{sec:Heegner}. Assuming $\mathbb{E}_0^\pm\neq 0$, a recent result of Matar (\cite{Matar}) ensures that $\mathfrak{X}_p^\pm(E/K_\infty)$ is free of rank $1$ over $\Lambda$ and the group of universal norms is generated by a Heegner point in $E(K)$, which is non-trivial modulo $p$ by assumption: this provides another proof of the main theorem of \cite{HLV2022}, but notice once again that this approach is at least problematic in the inert plus case, as $\mathbb{E}_0^+=0$. 

\subsection*{Notation and conventions}

We choose algebraic closures $\overline\Q$ of $\Q$ and $\overline\Q_p$ of $\Q_p$, where $p$ is a prime number. Moreover, we fix embeddings $\overline\Q\hookrightarrow\C$ and $\overline{\Q}\hookrightarrow\overline\Q_p$ for all $p$. Finally, if $L$ is a number field, then $G_L$ denotes the absolute Galois group of $L$ (with respect to a fixed algebraic closure of $L$).

\subsection*{Acknowledgements}

Jishnu Ray gratefully thanks Jeffrey Hatley for his encouragement and some interesting online conversations on signed Selmer groups of elliptic curves. He also thanks Antonio Lei for introducing him to the beautiful subject of supersingular Iwasawa theory during postdoc years. The authors would like to thank the anonymous referee for several valuable comments and suggestions.

\section{Signed Selmer groups of elliptic curves}

\subsection{Selmer groups and control theorem} \label{selmer-subsec}

Recall $\Lambda=\Z_p[\![\Gamma]\!]$. Write $\varepsilon:G_K\rightarrow \Lambda^\times$ for the (tautological) character given by composing the projection $G_K\rightarrow\Gamma$ with the canonical inclusion $\Gamma\hookrightarrow\Lambda^\times$. Let $T$ be the $p$-adic Tate module of $E$. The $\Lambda$-module $\mathbf{T}\defeq T\otimes_{\Z_p}\Lambda(\varepsilon^{-1})$ is free of rank $2$, compact and equipped with a continuous action of $G_\Q$. Define the discrete $\Lambda$-module $\mathbf{A}\defeq\Hom(\mathbf{T}^\iota,\Bmu_{p^\infty})$, where $\Bmu_{p^\infty}\subset\overline{\Q}$ is the group of all $p$-power roots of unity and $x\mapsto x^\iota$ denotes the canonical involution of $\Lambda$. Shapiro's lemma shows that there are isomorphisms 
\[ H^1(K,\mathbf{A})\simeq\varinjlim_{n\geq 1,m\geq0}H^1\bigl(K_m,E[p^n]\bigr) \]
and
\[ H^1(K,\mathbf{T})\simeq\varprojlim_m\varprojlim_nH^1(K_m,T/p^nT). \]
Following \cite[\S5]{BerLongVen} and \cite[Section 4]{BBLII}, we may define discrete Selmer groups $\Sel^\pm(K,\mathbf{A})$ and, by propagation, Selmer groups $\Sel^\pm(K,\mathbf{A}[\mathfrak{P}])$ for every ideal $\mathfrak{P}$ of $\Lambda$, where $\mathbf{A}[\mathfrak P]$ stands for the $\Lambda$-submodule of $\mathbf{A}$ consisting of all the elements that are annihilated by $\mathfrak P$. The reader is referred to \cite{BerLongVen} and \cite{BBLII} for the definition of the $\pm$-local conditions at $p$, which depend on the behaviour of $p$ in $K$ (in the split case, see also \cite{HLV2022}). 

Let us fix once and for all a topological generator $\gamma$ of $\Gamma$ and for all integers $m\geq0$ set $\omega_m\defeq\gamma^{p^m}-1\in\Gamma$. In this paper we are concerned with the groups 
\begin{equation} \label{CT}
\Sel_{p^n}^\pm(E/K_m)\defeq\Sel^\pm\bigl(K,\mathbf{A}[\omega_m,p^n]\bigr);
\end{equation} 
here $\mathbf{A}[\omega_m,p^n]$ is a shorthand for $\mathbf{A}\bigl[(\omega_m,p^n)\bigr]$. Thanks to the irreducibility of $\bar\rho_{E,p}$ and assumption \eqref{Tam}, by \cite[Proposition 5.8]{BerLongVen} we know that for all $n\geq1$ and $m\geq1$ the following \emph{control theorem} holds: 
\begin{equation} \label{control-eq}
\Sel_{p^n}^\pm(E/K_{m-1})\simeq  \Sel_{p^n}^\pm(E/K_m)^{\Gal(K_m/K_{m-1})},\quad\Sel_{p^{n-1}}^\pm(E/K_m)\simeq\Sel_{p^n}^\pm(E/K_m)[p]. 
\end{equation}
Notice that the first isomorphism is induced by restriction in cohomology and the second 
by the canonical inclusion. 

\begin{remark}
Assumption \eqref{Tam} is used in the proof in \cite{BerLongVen} of the control theorem in \eqref{control-eq} to show that localizations at places not lying above $p$ are injective; for a discussion of the local terms at the prime(s) above $p$, see the proof of Proposition \ref{coreslemma} below, which makes crucial use of recent results of Burungale--Kobayashi--Ota (\cite{BKO2021}) to get a perfect local control theorem. When put together, these pieces of local information imply the surjectivity of the restriction $\res_{K_m/K_{m-1}}$. On the other hand, the irreducibility of $\bar\rho_{E,p}$ ensures that $E(K_m)[p^n]$ is trivial for all $m,n$ (see, \emph{e.g.}, \cite[Lemma 4.3]{Gross}), which shows that $\res_{K_m/K_{m-1}}$ is injective as well.
\end{remark}

There is an equality $\Sel^\pm(K,\mathbf{A})=\varinjlim_{m,n}\Sel^\pm_{p^n}(E/K_m)$, the direct limit being computed with respect to the canonical inclusions $E[p^n]\subset E[p^{n+1}]$ and the restriction maps. Let 
\[ \mathfrak{X}_p^\pm(E/K_\infty)\defeq\Hom\bigl(\Sel^\pm(K,\mathbf{A}),\Q_p/\Z_p\bigr) \]
denote the Pontryagin dual of $\Sel^\pm(K,\mathbf{A})$, which we assume to be a compact $\Lambda$-module of rank $1$ (\emph{cf.} \eqref{rank} and the discussion in the introduction of the paper which lists, to the best knowledge of the authors, the various set of hypothesis when this is true). Let us define 
\[ \Sel^\pm_{p^\infty}(E/K_m)\defeq\varinjlim_n\Sel^\pm_{p^n}(E/K_m) \]
and
\[ S_p^\pm(E/K_m)=\Ta_p\bigl(\Sel_{p^\infty}^\pm(E/K_m)\bigr)\defeq\Hom\bigl(\Q_p/\Z_p,\Sel^\pm_{p^\infty}(E/K_m)\bigr), \]
the direct limit being taken, as before, with respect to the maps induced by the inclusions $E[p^n]\subset E[p^{n+1}]$. 

\subsection{Selmer groups and corestriction} 

Quite generally, for a finite Galois extension $\mathcal L/\mathcal K$ of fields let 
\[ \tr _{\mathcal L/\mathcal K}\defeq\sum_{\sigma\in\Gal(\mathcal L/\mathcal K)}\sigma \] 
be the corresponding trace operator.

Let $A\defeq E[p^\infty]$ be the $p$-divisible group of $E$. For integers $m'\geq m\geq 0$, let 
\[\mathrm{cores}_{K_{m'}/K_m}:H^1(K_{m'},A)\longrightarrow H^1(K_m,A)\] 
be the corestriction map.

\begin{proposition} \label{coreslemma} 
The map $\mathrm{cores}_{K_{m'}/K_m}$ induces a map 
\[ \mathrm{cores}_{K_{m'}/K_m}:\Sel_{p^n}^\pm(E/K_{m'})\longrightarrow \Sel_{p^n}^\pm(E/K_{m}). \]    
\end{proposition}

\begin{proof}
This is a local question: we need to check that the corestriction map takes the local conditions defining Selmer groups over $K_{m'}$ to those defining Selmer groups over $K_m$. For primes outside $p$ this is standard, so let us only show this result for primes above $p$. We first recall the definition of finite plus and minus conditions. 

For $0 \leq m \leq \infty$, let $k_m$ be the localization of $K_m$ at a prime above $p$ (there are two such primes if $p$ splits in $K$, whereas there is a unique such prime if $p$ is inert in $K$, as we have assumed that $p$ does not divide the class number of $K$). Let $ O $ be the valuation ring of the local field $k_0$, which is $\Q_p$ if $p$ splits in $K$ and is the unique unramified quadratic extension of $\Q_p$ if $p$ is inert in $K$. Let $\mathfrak{m}_{k_m}$ be the maximal ideal of the valuation ring $O_m$ of $k_m$. 
Let $\widehat{E}$ be the formal group of $E/k_0$, which is a Lubin--Tate formal group 
for the uniformizer $-p\in O$. Using the formal group logarithm of $\widehat{E}$ twisted by finite order characters $\chi$ of $\Gamma$, we can define subgroups $\widehat{E}^\pm(\mathfrak{m}_{k_\infty})\subset\widehat{E}(\mathfrak{m}_{k_\infty})$ 
satisfying vanishing conditions that depend on the parity of the exponent $m$ of the $p$-power conductor $p^m$ of $\chi$. Then we define $H^1_{\mathrm{fin},\pm}(k_\infty,A)$ to be the image of $\widehat{E}^\pm(\mathfrak{m}_{k_\infty})\otimes_{\Z_p}\Q_p/\Z_p$ under the local Kummer map. Analogously, we can introduce $H^1_{\mathrm{fin},\pm}\bigl(k_m,A[p^n]\bigr)$. By Shapiro's lemma, $H^1(k_\infty,A)\simeq H^1(k,\mathbf{A})$. Define $H^1_{\mathrm{fin},\pm}(k,\mathbf{A})$ to be the image of $H^1_{\mathrm{fin},\pm}(k_\infty,A)$ under this isomorphism. For an ideal $\mathfrak{P}$ of $\Lambda$, let $H^1_{\mathrm{fin},\pm}\bigl(k,\mathbf{A}[\mathfrak{P}]\bigr)$ be the subgroup of $H^1\bigl(k,\mathbf{A}[\mathfrak{P}]\bigr)$ obtained by propagation from  $H^1_{\mathrm{fin},\pm}(k,\mathbf{A})$ under the canonical map $\mathbf{A}[\mathfrak{P}]\hookrightarrow \mathbf{A}$. By \cite[Proposition 5.4]{BerLongVen}, $H^1_{\mathrm{fin},\pm}(k,\mathbf{A})$ is a cofree $\Lambda\otimes_{\Z_p}O$-module of rank $1$ and,  by \cite[Proposition 5.3]{BerLongVen}, there is a canonical isomorphism (called \emph{local control theorem}) $H^1_{\mathrm{fin},\pm}\bigl(k,\mathbf{A}[\mathfrak{P}]\bigr)\simeq H^1_{\mathrm{fin},\pm}(k,\mathbf{A})[\mathfrak{P}]$. It is worth remarking that this last isomorphism is a consequence of the results in \cite{BKO2021}. 

Therefore, the local conditions defining $\Sel_{p^n}^\pm(E/K_m)$ at the primes above $p$ 
are those that correspond to $H^1_{\mathrm{fin},\pm}\bigl(k,\mathbf{A}[\omega_m,p^n]\bigr)$; by Shapiro's lemma, this group is a subgroup of $H^1\bigl(k_m,A[p^n]\bigr)$. There is a commutative triangle 
\[ \xymatrix@R=40pt@C=30pt{H^1(k_{m+1},A[p^n])\ar[drr]^-{\tr_{k_{m+1}/k_m}} \ar[rr]^-{\mathrm{cores}_{k_{m+1}/k_m}} & &H^1(k_{m},A[p^n])\ar[d]^-{\mathrm{res}_{k_{m+1}/k_m}}\\
&&H^1(k_{m+1},A[p^n])} \]
(see, \emph{e.g.}, \cite[Proposition 5.9]{Brown}), where, as above, $\tr_{k_{m+1}/k_m}$ denotes the trace operator. Set $\mathcal{G}\defeq\Gal(k_{m+1}/k_m)$. By the local control theorem, restriction induces an isomorphism 
\begin{equation} \label{local-iso-eq}
H^1_{\mathrm{fin},\pm}\bigl(k_m,A[p^n]\bigr)\overset\simeq\longrightarrow H^1_{\mathrm{fin},\pm}\bigl(k_{m+1},A[p^n]\bigr)^\mathcal{G},
\end{equation}
so the corestriction map is just the trace map composed with the inverse of \eqref{local-iso-eq}. Then it suffices to check that $H^1_{\mathrm{fin},\pm}\bigl(k_{m+1},A[p^n]\bigr)$ is invariant under the trace. As remarked above, $H^1_{\mathrm{fin},\pm}\bigl(k_{m+1},A[p^n]\bigr)$ can be naturally identified with 
$H^1_{\mathrm{fin},\pm}\bigl(k,\mathbf{A}[\omega_{m+1},p^n]\bigr)$; moreover, the trace map coincides with the multiplication-by-$(\omega_{m+1}/\omega_m)$ map. Finally, the local control theorem yields an identification $H^1_{\mathrm{fin},\pm}\bigl(k,\mathbf{A}[\omega_{m+1},p^n]\bigr)=H^1_{\mathrm{fin},\pm}\bigl(k,\mathbf{A}[p^n]\bigr)[\omega_{m+1}]$, 
and then the desired invariance is obvious. \end{proof}

By Proposition \ref{coreslemma}, corestriction induces maps $\mathrm{cores}_{K_{m'}/K_m}:S_p^\pm(E/K_{m'})\rightarrow S_p^\pm(E/K_{m})$ and we may define 
\begin{equation} \label{hat-S-eq}
\widehat{S}_p^\pm(E/K_\infty)\defeq\varprojlim_mS_p^\pm(E/K_m),
\end{equation}
where the inverse limit is taken with respect to the corestriction maps. Both $\mathfrak{X}_p^\pm(E/K_\infty)$ and $\widehat{S}_p^\pm(E/K_\infty)$ are $\Lambda$-modules, and there is an isomorphism 
\[\widehat{S}_p^\pm(E/K_\infty)\simeq\Hom_\Lambda\bigl(\mathfrak{X}_p^\pm(E/K_\infty),\Lambda\bigr); \]
in particular, $\widehat{S}_p^\pm(E/K_\infty)$ is a free $\Lambda$-module of rank $1$. 

\begin{definition}\label{def:uninorm} 
Let $m\in\N$. The \emph{universal norm submodule} of $S^\pm_p(E/K_m)$ is 
\[ US^\pm_p(E/K_m)\defeq\bigcap_{m'\geq m}\mathrm{cores}_{K_{m'}/K_m}\bigl(S_p^\pm(E/K_{m'})\bigr)\subset S_p^\pm(E/K_m). \]
\end{definition}

Recall that we are assuming \eqref{Norm}.

\begin{theorem} \label{THM} 
The $\Lambda$-module $\mathfrak{X}^\pm_{p}(E/K_\infty)$ has no non-trivial finite $\Lambda$-submodules if and only if the $\Z_p$-module $S_p^\pm(E/K)\big/US_p^\pm(E/K)$ is torsion-free. 
\end{theorem}

\begin{proof} The proof can be adapted from the ordinary case, which is treated in \cite[Section 6]{BerBor}: see, especially, \cite[Theorem 6.1]{BerBor} and \cite[Corollary 6.2]{BerBor}. As is observed in \cite{HLV2022}, the arguments are of a cohomological nature and do not involve local conditions at $p$. \end{proof}

\section{$\Lambda$-modules of Heegner points}

\subsection{Mordell--Weil groups} 

For all integers $m\geq0$, recall the elements $\omega_m\in\Gamma$ from \S \ref{selmer-subsec}. Let $\Phi_{m+1}(T)=\sum_{j=0}^{p-1}T^{j\cdot p^m}\in\Z[T]$ be the $p^{m+1}$-th cyclotomic polynomial. Finally, let us denote by $n_p$ the number of primes of $K$ lying above $p$ and put $\nu_p\defeq n_p-1\in\{0,1\}$. Now define the elements $\omega_m^\pm\in\Gamma$ as follows: 
\begin{enumerate}
\item[(a)] $\omega_0^+=\omega_1^+\defeq(\gamma-1)^{\nu_p}$; 
\item[(b)] $\omega_m^+\defeq(\gamma-1)^{\nu_p}\cdot\prod_{1\leq j\leq \lfloor \frac{m}{2}\rfloor }\Phi_{2j}(\gamma)$ for $m\geq2$; 
\item[(c)] $\omega_0^-\defeq\gamma-1$; 
\item[(d)] $\omega_m^-=(\gamma-1)\cdot\prod_{1\leq j\leq \lfloor \frac{m+1}{2}\rfloor}\Phi_{2j-1}(\gamma)$ for $m\geq1$.  
\end{enumerate}
Set $E^\pm(K_m)\defeq E(K_m)[\omega_m^\pm]$; in other words, $E^\pm(K_m)$ is the subgroup of $E(K_m)$ consisting of the points that are killed by $\omega_m^\pm$. Furthermore, define 
\[
\mathbb{M}^\pm_\infty\defeq\bigcup_{m\geq0}E^\pm(K_m)\otimes_\Z\Q_p/\Z_p,\quad\mathbb{M}^\pm_m\defeq (\mathbb{M}^\pm_\infty)^{\Gal(K_\infty/K_m)}. \]
Observe that $E^\pm(K)=E(K)$ if $p$ splits in $K$, whereas $E^-(K)=E(K)$ and $E^+(K)=0$ if $p$ is inert in $K$. 

\begin{lemma}\label{lem:norm} 
For all $m\geq1$, there is an inclusion $\tr_{K_m/K_{m-1}}\bigl(E^\pm(K_m)\bigr)\subset E^\pm(K_{m-1})$.
\end{lemma}

\begin{proof} For every $x\in E(K_m)$, there are equalities
\[   
\tr_{K_m/K_{m-1}}(x)=\frac{\omega_m}{\omega_{m-1}}x = \frac{\omega_m^+\omega_m^-}{\omega_{m-1}^+\omega_{m-1}^-}x=
     \begin{cases}
       \frac{\omega_m^+}{\omega_{m-1}^+}x &\quad\text{if $m$ is even},\\[3mm]
       \frac{\omega_m^-}{\omega_{m-1}^-}x &\quad\text{if $m$ is odd}.   
     \end{cases}
\]
To deduce them, one simply observes that $\omega_m^+=\omega_{m-1}^+$ if $m\geq 1$ is odd  and $\omega_m^-=\omega_{m-1}^-$ if $m\geq 2$ is even. Now suppose that $x \in E^+(K_m)$. If $m$ is even, then $\omega_{m-1}^+\cdot\bigl(\frac{\omega_m^+}{\omega_{m-1}^+}x\bigr)=0$ by definition and the result follows. If $m$ is odd, then $\omega^+_m=\omega_{m-1}^+$ and the result is obviously true. Finally, assume that $x \in E^-(K_m)$. If $m$ is odd, then $\omega_{m-1}^-\cdot\bigl(\frac{\omega_m^-}{\omega_{m-1}^-}x\bigr)=0$, while if $m$ is even, then it is enough to notice that $\omega_m^-=\omega_{m-1}^-$. \end{proof}

\subsection{Heegner points}\label{sec:Heegner} 

Following \cite[\S2.4]{BD96}, for all $m\geq 0$ let us choose a Heegner point $\tilde{z}_m\in E(H_{p^m})$ in such a way that the sequence ${(\tilde{z}_m)}_{m\geq 0}$ is compatible with respect to the trace operators. As in \cite[\S 2.5]{BerLongVen}, for all $m\geq0$ define
\[ z_m\defeq\sum_{\sigma\in\Delta}\tilde{z}_{m+1}^\sigma\in E(K_m).\]
We also define the two points  
\[z_{-1}\defeq\sum_{\sigma\in\Delta}\tilde{z}_0^\sigma\in E(K),\quad z_K\defeq\tr_{H_K/K}(\tilde{z}_{0})\in E(K). \] 
Observe that $z_0$, $z_{-1}$, $z_K$ all lie in $E(K)$; furthermore, $z_{-1}$ and $z_K$ satisfy the relation 
\[ z_{-1}=\frac{p-\varepsilon_K(p)}{u_K}\cdot z_K, \]
where $\varepsilon_K$ is the Dirichlet character attached to $K$ (thus, $\varepsilon_K(p)=1$ if $p$ splits in $K$ and $\varepsilon_K(p)=-1$ if $p$ is inert in $K$). These points satisfy the formulas
\begin{enumerate}
\item $\tr_{K_m/K_{m-1}}(z_m)=-z_{m-2}$ for $m\geq1$;
\item $z_0=\begin{cases}0 & \text{if $p$ is inert in $K$},\\[3mm]-2z_K & \text{if $p$ splits in $K$}.\end{cases}$ 
\end{enumerate}
The \emph{plus/minus Heegner points} $z_m^\pm$ are defined as
\[ z^{+}_{m} \defeq \begin{cases} z_{m} & \text{if $m$ is even}\\[3mm] z_{m-1} & \text{if $m$ is odd} \end{cases},\quad z^{-}_{m} \defeq \begin{cases} z_{m-1} & \text{if $m$ is even}\\[3mm] z_m & \text{if $m$ is odd} \end{cases}.
\]
By construction, $z_m^\pm\in E^\pm(K_m)$. A direct computation shows that these points satisfy the following relations: 
\begin{enumerate} 
\item[(a)] $\tr_{K_m/K_{m-1}}(z_m^+)=-z^+_{m-2}$ for all even $m \geq 2$;
\item[(b)] $\tr_{K_m/K_{m-1}}(z_m^+)=pz^+_{m-1}$ for all odd $m \geq 1$;
\item[(c)] $\tr_{K_m/K_{m-1}}(z_m^-)=pz^-_{m-1}$ for all even $m \geq 2$;
\item[(d)] $\tr_{K_m/K_{m-1}}(z_m^-)=-z^-_{m-2}$ for all odd $m \geq 1$.
\end{enumerate}
Note that we also have 
\begin{enumerate} 
\item[(e)] $\tr_{K_1/K}(z_1^+)=\begin{cases}0 & \text{if $p$ is inert in $K$},\\[3mm]-2pz_K & \text{if $p$ splits in $K$};\end{cases}$
\item[(f)] $\tr_{K_1/K}(z_1^-)=\displaystyle{\frac{\varepsilon_K(p)-p}{u_K}}\cdot z_K$.
\end{enumerate}

\subsection{$\Lambda$-modules of Heegner points}\label{sec:Heegner_points}

Using the trace formulas for Heegner points in \S\ref{sec:Heegner}, 
we see that $z_m^\pm\in E^\pm(K_m)$. 
As in \cite[Section 4.4]{LV2019}, define $\calE^\pm_{m,n}$ to be the $\Lambda$-submodule of $\Sel_{p^n}^\pm(E/K_m)$ generated by $z_m^\pm$. Note that $\calE^+_{0,n}$ is trivial for all $n\geq 1$ when $p$ is inert in $K$, as $z_m^+=0$ in this case by definition. Then define the (discrete) $\Lambda$-submodule 
\[ \calE^\pm_\infty\defeq\varinjlim_{m,n} \calE_{m,n}^\pm \subset \Sel_{p^\infty}^\pm(E/K_\infty). \]
In general, denote by $M\mapsto M^\vee$ Pontryagin duality. In particular, let $\mathfrak{Z}_\infty^\pm\defeq(\mathcal{E}_\infty^\pm)^\vee$ be the Pontryagin dual of $\calE_\infty^\pm$.

\begin{proposition}\label{propfreeness}
The $\Lambda$-module $\mathfrak{Z}_\infty^\pm$ is free of rank $1$.
\end{proposition}

\begin{proof} One can closely mimic the proof of \cite[Proposition 4.7]{LV2019} (see also the proof of \cite[Proposition 4.5]{HLV2022}), so we only offer a sketch of the arguments. First of all, the canonical surjections $\Lambda_{m,m}\twoheadrightarrow\mathcal{E}^\pm_{m,m}$ induce injections $(\mathcal{E}^\pm_{m,m})^\vee\hookrightarrow \Lambda_{m,m}^\vee$; using the isomorphism between $\Lambda_{m,m}$ and its Pontryagin dual, we obtain an injection $\mathfrak{Z}_\infty^\pm\hookrightarrow \Lambda$. Results of Cornut (\cite{Cor02}) ensure that the point $z_m^\pm$ is non-torsion for $m$ sufficiently large; combined with the norm relations in \S\ref{sec:Heegner}, this shows that the module $\mathcal{E}_\infty^\pm$ is non-zero (\emph{cf.} \cite[Lemmas 4.5 and 4.6]{LV2019}). Finally, the injection $\mathfrak{Z}_\infty^\pm\hookrightarrow \Lambda$ implies that either $\mathfrak{Z}_\infty^\pm=0$ or $\mathfrak{Z}_\infty^\pm\simeq\Lambda$, so the claim of the proposition follows by Pontryagin duality from the non-triviality of $\mathcal{E}_\infty^\pm$. \end{proof}

For all integers $m\geq 0$ and $n\geq 1$, define the finite $\Lambda$-module 
\begin{equation} \label{defE}
\mathbb{E}_{m,n}^\pm\defeq(\mathcal{E}_\infty^\pm)^{\Gal(K_\infty/K_m)}[p^n];
\end{equation}
for notational convenience, set $\mathbb{E}_m^\pm\defeq\mathbb{E}_{m,1}^\pm$. There are inclusions of $\Lambda$-modules 
\begin{equation}\label{inclusions}
\mathcal{E}_{m,n}^\pm\subset\mathbb{E}_{m,n}^\pm\subset \mathbb{M}_m^\pm[p^n]\subset \Sel_{p^n}^\pm(E/K_m).
\end{equation}
Now set $G_m\defeq\Gal(K_m/K)$, $\mathcal{G}_m\defeq\Gal(K_\infty/K_m)$ and $R_m\defeq\mathbb{F}_p[G_m]$. 

The next lemma offers extensions of results from \cite{Ber-Annihilator}, \cite{Bertolini95}, \cite{BerBor}.

\begin{lemma}\label{HLVLEMMA}
\begin{enumerate}
\item The $\Lambda$-module $\mathbb{E}_{m,n}^\pm$ is cyclic. 
\item Restriction gives an injection of $\Lambda$-modules $\mathbb{E}_{m,n}^\pm \hookrightarrow \mathbb{E}_{m+1,n}^\pm $ inducing an isomorphism $\mathbb{E}_{m,n}^\pm\simeq (\mathbb{E}_{m+1,n}^\pm)^{\Gal(K_{m+1}/K_m)}$. By an abuse of notation, we shall view these canonical maps as an inclusion and an equality, respectively.
\item If $\mathbb{E}_{m}^\pm\neq 0$, then $\Sel_{p}^\pm(E/K_m)$ admits a free $R_{m}$-module $U_m^\pm$ of rank $1$ such that $\mathbb{E}_{m}^\pm\subset U_m^\pm\subset\mathbb{E}_{m+1}^\pm$. 
    %{\color{red} \item If $\mathfrak{C}_{m}^\pm\neq 0$ then $\Sel_{p}^\pm(E/K_m)\simeq U_m^\pm\oplus T_m^\pm$ where $T_m^\pm$ is finite of order bounded independently of $m$. }
\end{enumerate}   
\end{lemma}

\begin{proof} The proofs of parts (1), (2) and (3) are analogous to those of 
\cite[Corollary 4.6]{HLV2022}, \cite[Lemma 4.8]{HLV2022} and \cite[Lemma 5.10]{HLV2022}, respectively. For the reader's convenience, we review the arguments. 

(1) The $\Lambda$-module $\mathbb{E}_{m,n}^\pm$ is finite, hence isomorphic to its Pontryagin dual, which is a quotient of the cyclic module $\Lambda$-module $\mathfrak{Z}_\infty$. It follows that $\mathbb{E}_{m,n}^\pm$ is cyclic. 

(2) It follows from the control theorem in \eqref{control-eq} that restriction induces an injection 
\begin{equation} \label{inj-cont-eq}
\mathbb{E}^\pm_{m,n} \longmono \bigl({\mathbb{E}_{m+1,n}^\pm}\bigr)^{\Gal(K_{m+1}/K_m)}.
\end{equation}
On the other hand, the equalities 
\[ (\mathcal{E}_\infty^\pm)^{\mathcal{G}_m}=\bigl((\mathcal{E}_\infty^\pm)^{\mathcal{G}_{m+1}}\bigr)^{\Gal(K_{m+1}/K_m)} \]
and 
\[ \bigl((\mathcal{E}_\infty^\pm)^{\mathcal{G}_{m+1}}\bigr)^{\Gal(K_{m+1}/K_m)}[p^n]=
\bigl((\mathcal{E}_\infty^\pm)^{\mathcal{G}_{m+1}}[p^n]\bigr)^{\Gal(K_{m+1}/K_m)} \]
yield the surjectivity of \eqref{inj-cont-eq}.

(3) For $r\in\{m,m+1\}$, fix a generator $\gamma_r$ of $G_r$ and set $t_r^\pm\defeq p^r-\dim_{\mathbb{F}_p}(\mathbb{E}_r^\pm)$. By part (1) and \cite[Lemma 3]{Ber-Annihilator}, for $r\in\{m,m+1\}$ there is an isomorphism $\mathbb{E}_r^\pm\simeq R_r/(\gamma_r-1)^{p^r-t_r^\pm}$; there is also an isomorphism of $R_r$-modules $R_r/(\gamma_r-1)^{p^r-t_r^\pm}\simeq (\gamma_r-1)^{t_r^\pm}R_r$ and we conclude that there is an isomorphism of $R_r$-modules 
\begin{equation}\label{lemmaHLV1}
\mathbb{E}_r^\pm\simeq(\gamma_r-1)^{t_r^\pm}R_r.
\end{equation}
The group $\Gal(K_{m+1}/K_m)\simeq\big\langle\gamma_m^{p^{m+1}}\big\rangle$ is cyclic of order $p$ and the coefficient ring of $R_m$ is ${\mathbb F}_p$, so for each integer $s\in\{0,\dots,p^{m+1}\}$ we have
\begin{equation}\label{lemmaHLV3}
\begin{split}
\bigl((\gamma_{m+1}-1)^s R_{m+1}\bigr)^{\Gal(K_{m+1}/K_m)}&\simeq \bigl(1+\gamma_{m+1}^{p^m}+\dots+\gamma_{m+1}^{p^m(p-1)}\bigr)\cdot(\gamma_{m+1}-1)^s R_{m+1}\\
&=(\gamma_{m+1}-1)^{p^{m+1}-p^m+s} R_{m+1}.
\end{split}
\end{equation}
In particular, we conclude that 
\begin{equation}\label{lemmaHLV}
\dim_{\mathbb{F}_p}\Bigl(\bigl((\gamma_{m+1}-1)^s R_{m+1}\bigr)^{\Gal(K_{m+1}/K_m)}\Bigr)=p^m-s.
\end{equation} 
By part (2), $\dim_{\mathbb{F}_p}(\mathbb{E}^\pm_m)=\dim_{\mathbb{F}_p}\bigl((\mathbb{E}_{m+1}^\pm)^{\Gal(K_{m+1}/K_m)}\bigr)$. Combining \eqref{lemmaHLV1} for $r=m+1$ and \eqref{lemmaHLV} for $s=t_{m+1}^\pm$ gives $\dim_{\mathbb{F}_p}\bigl((\mathbb{E}_{m+1}^\pm)^{\Gal(K_{m+1}/K_m)}\bigr)=p^m-t_{m+1}^\pm$. On the other hand, $\dim_{\mathbb{F}_p}(\mathbb{E}_m^\pm)=p^m-t_m^\pm$ by \eqref{lemmaHLV1} with $r=m$, so $t_m^\pm=t_{m+1}^\pm$. Put $t^\pm\defeq t_m^\pm$ and define the $R_{m+1}$-module 
\begin{equation}\label{lemmaHLV2}
U_m^\pm\defeq(\gamma_{m+1}-1)^{p^{m+1}-p^m-t^\pm}\mathbb{E}_{m+1}^\pm\simeq (\gamma_{m+1}-1)^{p^{m+1}-p^m}R_{m+1},
\end{equation} 
where the isomorphism follows from \eqref{lemmaHLV1} with $r=m+1$. The group $U_m^\pm$, which is contained in $\mathbb{E}_{m+1}^\pm$ by definition, is isomorphic to 
$R_{m+1}^{\Gal(K_{m+1}/K_m)}$ as an $R_{m+1}$-module (take $s=0$ in \eqref{lemmaHLV3}), so it is invariant under the action of $\Gal(K_{m+1}/K_m)$. In particular, $U_m^\pm$ is equipped with a canonical $R_m$-module structure. Taking $s=t^\pm=t_{m+1}^\pm$ in \eqref{lemmaHLV3}, and using \eqref{lemmaHLV1} with $r=m+1$, we see that there is an isomorphism of $R_{m+1}$-modules 
\[ (\mathbb{E}_{m+1}^\pm)^{\Gal(K_{m+1}/K_m)}\simeq(\gamma_{m+1}-1)^{p^{m+1}-p^m+t^\pm}R_{m+1}. \]
The $R_{m+1}$-module $(\gamma_{m+1}-1)^{p^{m+1}-p^m+t^\pm}R_{m+1}$ is a submodule of $(\gamma_{m+1}-1)^{p^{m+1}-p^m}R_{m+1}$ and there is an isomorphism
$(\gamma_{m+1}-1)^{p^{m+1}-p^m}R_{m+1}\simeq U_m^\pm$ of $R_{m+1}$-modules. 
Thus, we conclude that $U_m^\pm$ contains $(\mathbb{E}_{m+1}^\pm)^{\Gal(K_{m+1}/K_m)}$.
Moreover, $(\mathbb{E}_{m+1}^\pm)^{\Gal(K_{m+1}/K_m)}\simeq \mathbb{E}_m^\pm$ by part (2), 
so $U_m^\pm$ containes $\mathbb{E}_m^\pm$. On the other hand, $U_m^\pm$ 
is a subgroup of $\Sel_p^\pm(E/K_{m+1})$ that is invariant under the action 
of ${\Gal(K_{m+1}/K_m)}$, therefore, under the isomorphism  
\[ \Sel_p^\pm(E/K_{m+1})^{\Gal(K_{m+1}/K_m)}\simeq \Sel_p^\pm(E/K_m) \]
from \eqref{control-eq}, it is identified with a subgroup of $\Sel_p^\pm(E/K_m)$, denoted with the same symbol. Finally, $\dim_{\mathbb{F}_p}(U_m^\pm)=p^m$ by \eqref{lemmaHLV2}, so we conclude that $U_m^\pm$ is free of rank $1$ over $R_m$. \end{proof}

\subsection{Shafarevich--Tate groups} 

For all $m \geq 0$ and $n \geq 1$, consider the (global) Kummer maps 
\[ \kappa_{m,n}:E(K_m)/p^nE(K_m) \longmono H^1\bigl(K_m,E[p^n]\bigr) \]
and
\[ \kappa_m: E(K_m) \otimes_\Z \Q_p/\Z_p \longmono H^1\bigl(K_m,E[p^\infty]\bigr). \] 
Via these maps, we may identify $\mathbb{M}^\pm_m$ with a divisible subgroup of $\Sel_{p^\infty}^\pm(E/K_m)$ and $\mathbb{M}^\pm_m{[p^n]}$ with a subgroup of $\Sel_{p^\infty}^\pm(E/K_m)[p^n]\simeq\Sel_{p^n}^\pm(E/K_m)$, where this isomorphism, which we shall regard as an identification, follows from the control theorem in \eqref{control-eq}. Let us define \emph{$\pm$-Shafarevich--Tate groups} as 
\begin{alignat*}{2}
&\Sha_{p^n}^\pm(E/K_m)&&\defeq\Sel_{p^n}^\pm(E/K_m)\big/\mathbb{M}_m^\pm[p^n],\\
&\Sha_{p^\infty}^\pm(E/K_m)&&\defeq\varinjlim_{n\geq 1}\Sha^\pm_{p^n}(E/K_m),\\
&\Sha^\pm_{p^\infty}(E/K_\infty)&&\defeq\varinjlim_{m\geq 0}\Sha^\pm_{p^\infty}(E/K_m).
\end{alignat*}
Notice that $\mathbb{M}^+_0=0$ in the inert case, so $\Sha_{p^\infty}^+(E/K)=\Sel_{p^\infty}^+(E/K)$. On the other hand, in both the inert case and the split case, there is an equality $\mathbb{M}^-_0=E(K)\otimes_\Z\Q_p/\Z_p$, so $\Sel_{p^\infty}^-(E/K)=\Sel_{p^\infty}(E/K)$ and $\Sha_{p^\infty}^-(E/K)=\Sha_{p^\infty}(E/K)$, where 
\[ \Sha_{p^\infty}(E/K)\defeq\Sel_{p^\infty}(E/K)\big/\kappa_0\bigl(E(K)\otimes_\Z\Q_p/\Z_p\bigr) \]
is the usual $p$-primary Shafarevich--Tate group of $E$ over $K$.    

\subsection{Universal norms and finite submodules}

Recall that $G_m=\Gal(K_m/K)$, $R_m=\mathbb{F}_p[G_m]$ and $\mathbb{E}_m^\pm=\mathbb{E}_{m,1}^\pm$. We also recall from the introduction the conditions 
\begin{equation} 
\tag{$\Sha(m)$}  \text{$\Sha_{p^\infty}^\pm(E,K_m/K_{m+1})=0$}
\end{equation}
and 
\begin{equation}
\tag{$\mathrm{Heeg}(m)$}  \text{$\mathbb{E}_m^\pm\neq 0$},
\end{equation}
under which we prove the following auxiliary result.

\begin{lemma}
If $\Sha(m)$ and $\mathrm{Heeg}(m)$ hold true for an integer $m\geq 0$, then $U_m^\pm\subset \mathbb{M}_m^\pm[p].$
\end{lemma}

\begin{proof} We proceed as in the proof of \cite[Lemma 5.12]{HLV2022}. Set $G\defeq\Gal(K_{m+1}/K_m)$; there is a commutative diagram with exact rows 
\[\xymatrix{0\ar[r] & \mathbb{M}_m^\pm[p]\ar[d]\ar[r] & \Sel_p^\pm(E/K_m)\ar[r]\ar[d]^-\simeq & \Sha_p^\pm(E/K_m)\ar[d]\ar[r]&0\\
0\ar[r] & \mathbb{M}_{m+1}^\pm[p]^G\ar[r] & \Sel_p^\pm(E/K_{m+1})^G\ar[r] & \Sha_p^\pm(E/K_{m+1})^G,
} 
\] 
where the vertical middle isomorphism follows from the control theorem in \eqref{control-eq}. Since the middle vertical map is injective, the left vertical arrow is injective too; on the other hand, the right vertical arrow is injective by $(\Sha(m))$ 
and the middle vertical map is surjective, so the snake lemma shows that the 
left vertical arrow is surjective as well. Thus, we get an isomorphism
\[ \mathbb{M}_m^\pm[p]\overset\simeq\longrightarrow\mathbb{M}_{m+1}^\pm[p]^G, \]
which we shall view as an identification. By definition, $U_m^\pm\subset \mathbb{M}_{m+1}^\pm[p]^G$, and the result follows. \end{proof}

\begin{comment}
\begin{lemma}\label{HLVLEMMA1}
We have \[S_p^\pm(E/K)=\Ta_p(\mathbb{M}_0)\defeq\varprojlim_n \mathbb{M}_0^\pm[p^n].\]
\end{lemma}
\begin{proof}
It follows from the definition that $\mathbb{M}_0^\pm[p^n]=\kappa_{0,n}(E(K)/p^nE(K))$ and  
$\Sel^\pm_{p^n}(E/K)=\Sel_{p^n}(E/K)$, and therefore we have an exact sequence 
\[0\longrightarrow \mathbb{M}_0^\pm[p^n]\longrightarrow \Sel_{p^n}(E/K)\longrightarrow \Sha_{p^n}(E/K)\longrightarrow 0\] where $\Sha_{p^n}(E/K)=\Sha_{p^n}^\pm(E/K)$ is finite of order bounded independently of $n$. The result follows. 
\end{proof}
\end{comment}

Now we can prove Theorem A in the introduction. 

\begin{theorem} \label{THM1} 
Suppose that \eqref{Norm} holds and that $\Sha(m_0)$ and $\mathrm{Heeg}(m_0)$ are satisfied for an integer $m_0\geq 0$. We also assume \eqref{rank}, i.e., $\mathfrak{X}_p^\pm(E/K_\infty)$ has $\Lambda$-rank $1$. Then $US_p^\pm(E/K)$ is a free $\Z_p$-module of rank $1$ and the quotient $S_p^\pm(E/K)\big/US_p^\pm(E/K)$ is torsion-free. 
\end{theorem}

\begin{proof} Since $S_p^\pm(E/K_m)=\varprojlim_n \Sel_{p^\infty}^\pm(E/K_m)[p^n]$, the $\Z_p$-module $S_p^\pm(E/K_m)$ is free of finite rank; thus, $S_p^\pm(E/K)$ is a free $\Z_p$-module of finite rank. Since $US_p^\pm(E/K)\subset S_p^\pm(E/K)$ and $S_p^\pm(E/K)\simeq\Z_p^r$ for some integer $r\geq1$, the theorem follows if we show that $US_p^\pm(E/K)$ is isomorphic to $\Z_p$ and contains an element of $S_p^\pm(E/K)$ which is not divisible by $p$. The Pontryagin dual $\mathfrak{X}_p^\pm(E/K_\infty)$ of $\Sel_{p^\infty}^\pm(E/K_\infty)$ has $\Lambda$-rank equal to $1$, and therefore the rank of the $\Lambda$-module $\widehat{S}_p^\pm(E/K_\infty)$ from \eqref{hat-S-eq} is $1$. Write $\widehat{S}_p^\pm{(E/K_\infty)}_{G_\infty}$ for the $\Z_p$-module of $G_\infty$-coinvariants of $\widehat{S}_p^\pm(E/K_\infty)$; there is a canonical surjection 
\[ \widehat{S}_p^\pm{(E/K_\infty)}_{G_\infty}\longepi US_p^\pm(E/K), \] 
so the $\Z_p$-rank of $US_p^\pm(E/K)$ is at most $1$. 

Fix $m\geq m_0$. Then $\mathbb{E}_m^\pm\neq 0$. Consider the Tate module $\Ta_p(\mathbb{M}_m^\pm)\defeq\varprojlim_n \mathbb{M}_m^\pm[p^n]$ of $\mathbb{M}_m^\pm$ and recall that $\mathbb{M}_m^\pm[p]$ contains $U_m^\pm$, which is a free $R_m$-module. In light of the canonical isomorphism $\Ta_p(\mathbb{M}_m^\pm)\big/p\Ta_p(\mathbb{M}_m^\pm)\simeq\mathbb{M}_m^\pm[p]$, we can take a lift 
\[ \widetilde{U}_m^\pm\subset \Ta_p(\mathbb{M}_m^\pm) \] 
of $U_m^\pm$ modulo $p$ that is a free $\Z_p[G_m]$-module of rank $1$ and let $u_m$ be a generator of it. This module injects into a rank $1$ free $\Z_p[G_m]$-submodule of $S_p^\pm(E/K_m)$; we still denote this submodule by $\widetilde{U}_m^\pm$ and write $u_m\in\widetilde U_m^\pm$ for a generator. Then define 
\[ v_m\defeq\mathrm{cores}_{K_m/K}(u_m). \] 
Observe that $\tr_{K_m/K}(u_m)$ is not divisible by $p$ in $S_p^\pm(E/K_m)$, as $\widetilde{U}_m^\pm$ is a free $\Z_p[G_m]$-module. To justify this non-divisibility, assume by contradiction that $\tr_{K_m/K}(u_m)=pw$ for some $w\in S_p^\pm(E/K_m)$; then $w=\lambda u_m$ for some $\lambda\in \Z_p[G_m]$ and so
\[ \Biggl(\,\sum_{g\in G_m}g-p\lambda\Biggr)u_m=0. \] 
Since $u_m$ is a generator of the free $\Z_p[G_m]$-module $\widetilde{U}^\pm_m$, it follows that $\sum_{g\in G_m}g=p\lambda$, which is impossible because $\sum_{g\in G_m}g\neq 0$ in $R_m$, while the image of $p\lambda$ in $R_m$ is (of course) trivial. Now we check that $v_m$ is not divisible by $p$ in $S_p^\pm(E/K)$. By contradiction, suppose again that $v_m$ is divisible by $p$ in $S_p^\pm(E/K)$ and choose $w_m\in S_p^\pm(E/K)$ such that $v_m=pw_m$. As in the proof of Proposition \ref{coreslemma}, there is a commutative triangle
\[ \xymatrix@R=40pt{S_p^\pm(E/K_{m})\ar[drr]^-{\tr_{K_m/K}} \ar[rr]^-{\mathrm{cores}_{K_m/K}} && S_p^\pm(E/K)\ar[d]^-{\mathrm{res}_{K_m/K}}\\
&& S_p^\pm(E/K_{m})^{\Gal(K_m/K)}} \]
(see, \emph{e.g.}, \cite[Proposition 5.9]{Brown}). Thus, there are equalities 
\[ \tr_{K_m/K}(u_m)=\res_{K_m/K}(v_m)=p\cdot\mathrm{res}_{K_m/K}(w_m). \]
It follows that $\tr_{K_m/K}(u_m)$ is $p$-divisible in $S_p^\pm(E/K_m)$, which is a contradiction. Therefore, $v_m$ is an element of $S_p^\pm(E/K)$ that
\begin{itemize}
\item is not divisible by $p$;
\item is a corestriction from $K_m$.
\end{itemize}
Since $S_p^\pm(E/K)$ is compact, the sequence ${(v_m)}_{m\geq m_0}$ admits a subsequence ${(v_{n_i})}_i$ converging to an element $v_\infty\in S_p^\pm(E/K)$. Then there is a subsequence ${(v_{m_i})}_i$ of ${(v_m)}_{m\geq m_0}$ satisfying the following conditions: 
\begin{itemize}
    \item $v_\infty=v_{m_i}+p^i \epsilon_i$ for some $\epsilon_i\in  S_p^\pm(E/K)$; 
    \item $v_{m_i}=\tr_{K_i/K}(w_i)$ for some $w_i\in S_p^\pm(E/K_i)$. 
\end{itemize}
Hence, $v_\infty=\tr_{K_i/K}(w_{i}+\epsilon_i)$ for all $i\geq 1$, so 
$v_\infty\in US_p^\pm(E/K)$. Moreover, since the $v_{m_i}$ are not divisible by $p$ 
in $S_p^\pm(E/K)$, the same is true of $v_\infty$. In particular, $v_\infty$ is non-zero, 
so the $\Z_p$-rank of $US_p^\pm(E/K)$ is at least $1$; on the other hand, this rank is at most $1$, so it is equal to $1$. Now $US_p^\pm(E/K)$ has $\Z_p$-rank $1$ and contains an element of $S_p^\pm(E/K)$ that is not divisible by $p$. As we observed before, this implies the theorem. \end{proof}

As an application of our main theorem, we can prove our result on the non-existence of finite non-trivial $\Lambda$-submodules of Pontryagin duals.

\begin{corollary} 
Under the assumptions of Theorem \ref{THM1}, $\mathfrak{X}_{p}^\pm(E/K_\infty)$ does not have any finite non-trivial $\Lambda$-submodules. 
\end{corollary}

\begin{proof} A combination of Theorems \ref{THM} and \ref{THM1}. \end{proof}

\begin{remark}
In the inert plus case, the assumption $\mathbb{E}_m^+\neq 0$ for some $m\geq 0$ is equivalent to the condition $\mathbb{E}_m^+\neq 0$ for some $m\geq 1$, as $\mathbb{E}_0^+= 0$. This explains why the arguments in \cite{HLV2022} cannot be adapted to this case. 
\end{remark}

\bibliographystyle{amsplain}
\bibliography{main}
\end{document}